\UseRawInputEncoding
\documentclass[12pt,a4paper,reqno]{amsart}
\usepackage{cite}
\usepackage[margin=2.4cm]{geometry}
\usepackage{graphics, amsmath,enumitem, comment}
\usepackage{amsthm}
\usepackage{amssymb}
\usepackage{bbm}
\usepackage{mathrsfs}
\RequirePackage{color}
\RequirePackage[breaklinks]{hyperref}

\RequirePackage[usenames,dvipsnames]{xcolor}
\RequirePackage{pgf, tikz}
\usepackage{epsfig}
\def\bc{\begin{center}}
\def\ec{\end{center}}
\def\be{\begin{equation}}
\def\ee{\end{equation}}

\def\P{\mathbb P}
\def\G{\mathcal G}

\def\Q{\mathcal Q}
\def\R{\mathbb R}

\def\B{\mathcal B}
\DeclareMathOperator{\diam}{diam}
\numberwithin{equation}{section}
\newtheorem{lem}{Lemma}[section]
\newtheorem{dfn}{Definition}[section]
\newtheorem{pro}{Proposition}[section]
\newtheorem{thm}{Theorem}[section]

\newtheorem{cor}{Corollary}[section]
\theoremstyle{remark}
\newtheorem{rem}{\bf Remark}[section]

\begin{document}

\title{Large intersection property for limsup sets in metric space}
\author[Zhang-nan Hu]{Zhang-nan Hu}
\address{Zhang-nan Hu, School of Mathematics, South China University of Technology, Guangzhou, 510641, China}
\email{hnlgdxhzn@163.com}
\author{Bing Li}
\address{Bing Li, School of Mathematics, South China University of Technology, Guangzhou, 510641, China}
\email{scbingli@scut.edu.cn}
\author[Linqi Yang]{Linqi Yang*}\thanks{* Corresponding author}
\address{Linqi Yang, School of Mathematics, South China University of Technology, Guangzhou, 510641, China}
\email{2981502134@qq.com}

\begin{abstract} 
 We show that limsup sets generated by a sequence of open sets in compact Ahlfors $s$-regular space $(X,\mathscr{B},\mu,\rho)$ belong to the classes of sets with large intersections with index $\lambda$, denoted by $\mathcal{G}^{\lambda}(X)$, under some conditions. In particular, this provides a lower bound on Hausdorff dimension of such sets. These results are applied to obtain that limsup random fractals with indices $\gamma_2$ and $\delta$ belong to $\mathcal{G}^{s-\delta-\gamma_2}(X)$ almost surely, and random covering sets with exponentially mixing property belong to $\mathcal{G}^{s_0}(X)$ almost surely, where $s_0$ equals to the corresponding Hausdorff dimension of covering sets almost surely. We also investigate the large intersection property of limsup sets generated by rectangles in metric space.
\end{abstract}


\maketitle

\section{Introduction}

Sets with large intersection were introduced by Falconer in \cite{fal94}. Given $s\in(0,d]$, he defined $\G^s(\R^d)$ to be the class of all $G_{\delta}$ sets $F$ in  $\R^d$ such that  
\[\dim_{\rm H}\bigcap_{n=1}^{\infty} f_n(F) \ge s\]
holds for all sequences of similarity transformations $\{f_n\}_{n\ge1}$, where $\dim_{\rm H}$ denotes the Hausdorff dimension. It is known that $\G^s(\R^d)$ is the maximal class of $G_{\delta}$ sets satisfying (i) Hausdorff
dimension larger than $s$, and (ii) closed under similarity transformations and countable intersections. The packing dimension of sets in $\G^s(\R^d)$ is $d$ from the fact that they are dense $G_{\delta}$ sets. 
In \cite{fal94}, Falconer also gave several equivalent definitions and established various properties of $\G^s(\R^d)$. 
In 2004, Bugeaud \cite{bu04} generalized the class for more general gauge functions. 
Later, Negreira and Sequeira \cite{ns} extended the class of sets in Euclidean space with large intersection property to metric space endowed with doubling measure.

There are many applications of large intersection property. Lots of mathematicians applied this property for estimating Hausdorff dimension from below in the study of  Diophantine approximations, and refer to \cite{bu04}, \cite{ding}, \cite{fal94}, \cite{ns} for more details. 
 In 2007, Durand \cite{du07} investigated the size and large intersection properties of limsup sets generated by homogeneous ubiquitous systems in $\R^d$. In 2010, Durand \cite{du10} showed that random covering sets $\limsup\limits_{n\to\infty}B(\xi_n,r_n)$ are sets with large intersection, where $\{\xi_n\}_{n\ge1}$ are independent and uniformly distributed random variables on the circle. 
 In 2019, Persson \cite{per19} proved that dynamical covering sets $\limsup\limits_{n\to\infty}B(T^nx,n^{-\alpha})$ has large intersection property for dynamical systems $(T,\mu)$ which have summable decay of correlations. 
 Ding \cite{ding} showed that under a full Hausdorff measure assumption, the limsup sets generated by rectangles with some conditions in compact metric space are sets with large intersection. 
 In 2021, Persson \cite{per21} considered various sequences of open sets with general shapes, proved the corresponding limsup sets  have  large intersection properties, and obtained the lower bound on the Hausdorff dimension.  
 Aubry and Jaffard \cite{aj} noticed that large intersection property also occurred in probability theory, such as the multifractal analysis of random wavelet series. See also  \cite{dd, du09, fp} for more study in fractals, dynamical systems and the multifractal analysis of other stochastic processes.

Limsup sets, the upper limits of sequences of sets,  play an important role in many areas, such as random covering problem, shrinking target problem, the study of Brownian motion and so on. 
Motivated by the study on sets with large intersections, we are interested in the large intersection properties  of limsup sets generated by open sets in metric spaces $(X,\mathscr{B},\mu,\rho)$ and those applications.

\begin{dfn}\label{ar}
A Borel measure $\mu$ on metric space $(X,\rho)$ is 
Ahlfors $s$-regular $(0< s < \infty)$ if there exists a constant $1\le C <\infty$ such that 
\begin{equation}\label{eqah}
C^{-1}r^s\le\mu\big(B(x,r)\big)\le Cr^s
\end{equation} 
holds for all $x\in X$ and $0<r\le\diam X$, where $\diam X$ is the diameter of $X$. Here $B(x,r)=\{y\in X\colon \rho(x,y)<r\}$ is an open ball with centre $x$ and radius $r$. A space $(X,\mathscr{B},\mu,\rho)$ is said to be Ahlfors $s$-regular if $\mu$ satisfies formula (\ref{eqah}).
\end{dfn}

In this paper, we consider a probability space $(X,\mathscr{B},\mu,\rho)$ where $\mu$ is Ahlfors $s$-regular ($0<s<\infty$). For $B=B(x,r)$, we adopt the convention that $cB=B(x,cr)$, 
$B^t=B(x,r^{t/s})$ and $B=\varnothing$ if $r=0$.
\begin{dfn}\label{ene}
For $0\le t\le s$, the $t$-potential at $y\in X$ of the measure $\mu$ is defined as
\[\phi_t(\mu,y)=\int_X \rho(x,y)^{-t}d\mu(x).\]
The $t$-energy of the measure $\mu$ is defined as
\[I_t(\mu)=\int_X\int_X \rho(x , y)^{-t}d\mu(x)d\mu(y).\]
\end{dfn}
Denote
\[ \phi_t(\mu,U,y)=\int_U\rho(x,y)^{-t}d\mu(x),\]
\[I_t(\mu, U)=\int_U\int_U \rho(x , y)^{-t}d\mu(x)d\mu(y)\]
for $U\in \mathscr{B}$. 

\begin{thm}\label{main}
Let $(X,\mathscr{B},\mu,\rho)$ be a compact Ahlfors $s$-regular space $(0<s<\infty)$, and 
 $\{B_n\}_{n\ge1}$ be a sequence of balls in $X$ with $\diam B_n$ decreasing to 0 as $n\to\infty$. For $n\ge1$, $E_n$ is an open subset of $B_n$ and let
\[ \lambda=\sup\Big\{t\ge0\colon \sup_{n\ge1}\frac{I_t(\mu,E_n)\mu(B_n)}{\mu(E_n)^2}<\infty\Big\}.\]
Then $\mu\Big(\limsup\limits_{n\to\infty}B_n\Big)=1$ implies that $ \limsup\limits_{n\to\infty}E_n\in\mathcal{G}^{\lambda}(X).$
\end{thm}
\begin{cor}\label{main1}
Under the setting in Theorem \ref{main}, if $\mu\Big(\limsup\limits_{n\to\infty}B_n\Big)=1$, then we have $\dim_{\rm H}\Big(\limsup\limits_{n\to\infty}E_n\Big)\ge \lambda$ and $\dim_{\rm P}\Big(\limsup\limits_{n\to\infty}E_n\Big)=s$, where $\dim_{\rm P}$ denotes the packing dimension.
\end{cor}

When $\{E_n\}_{n\ge1}$ is a sequence of balls, an important tool in determining the Hausdorff dimension of $\limsup\limits_{n\to\infty}E_n$ is the Mass Transference Principle, which was established by Beresnevich and Velani \cite{bv06}.
\begin{thm}[Mass Transference Principle \cite{bv06}]\label{mtp}
Let $(X,\mathscr{B},\mu,\rho)$ be a locally compact Ahlfors $s$-regular space $(0<s<\infty)$,
Let $\{E_n\}_{n\ge1}$ be a sequence of balls in $X$ with $\diam E_n\to0$ as $n\to\infty$. 
Let $t>0$ and suppose that 
\[\mathcal{H}^s(\limsup_{n\to\infty}E_n^{t})=\mathcal{H}^s(X).\]
Then, 
\[\mathcal{H}^t(\limsup_{n\to\infty}E_n)=\mathcal{H}^t(X).\]
Here $\mathcal{H}^t (F)$ denotes the Hausdorff $t$ -measure of a set $F\subset X$. 
\end{thm}
For an Ahlfors $s$-regular space $(X,\mathscr{B},\mu,\rho)$, Heinonen \cite[Section 8.7]{Hein01} proved that $C'^{-1}\mathcal{H}^s\le \mu\le C'\mathcal{H}^s$, where $C'\ge1$ is a constant.  
Then by Theorem \ref{mtp}, given $t\in(0,s]$, we can deduce that $\dim_{\rm H}\Big(\limsup\limits_{n\to\infty}E_n\Big)\ge t$,  if $\mu\Big(\limsup\limits_{n\to\infty}E_n^{t}\Big)=1$. 
Such estimation can be derived by Corollary \ref{main1} when $X$ is compact. In addition, as the following corollary shown, $\limsup\limits_{n\to\infty}E_n$ also has large intersection property.

\begin{cor}\label{mainn}
Let $(X,\mathscr{B},\mu,\rho)$ be a compact Ahlfors $s$-regular space $(0<s<\infty)$, and  $\{B_n\}_{n\ge1}$ be a sequence of balls with $\diam B_n$ decreasing to 0 as $n\to\infty$.  For $t\in(0,s]$, if $\mu\Big(\limsup\limits_{n\to\infty}B_n^t\Big)=1$, then $ \limsup\limits_{n\to\infty}B_n\in\mathcal{G}^{t}(X).$
\end{cor}
Applying Theorem \ref{main} and Corollary \ref{mainn}, we prove that limsup random fractals, random covering sets and limsup sets generated by rectangles have large intersection property (Theorems 4.1, 4.7 and 4.8).

The rest of this paper is organized as follows. In Section 2 we give a brief review of the class of sets with large intersection
properties. In Sections 3, we give the proofs of  Theorem \ref{main} and Corollary \ref{mainn} which are our main results. In the last section, there are some examples. We apply our results to the study of the large intersection properties of  limsup random fractals under some conditions and random covering sets $\limsup\limits_{n\to\infty}B(\xi_n,r_n)$, where the centers $\{\xi_n\}_{n\ge1}$ are uniformly distributed and exponentially mixing random variables, are sets with large intersection properties. We also show limsup sets generated by rectangles have large intersection property.

\section{Preliminaries}
In this section, we refer $(X,\rho)$ to general metric space, and $\mu$ is a Borel measure on $X$. Denote $B(x,r)=\{y\in X\colon \rho(x,y)<r\}$ and $\overline{B}(x,r)=\{y\in X\colon \rho(x,y)\le r\}$.
\subsection{Generalized dyadic cubes in metric spaces}
\begin{dfn}\label{fd}
A metric space $(X, \rho)$ has the finite doubling property if any closed ball $\overline{B}(x, 2r) \subset X$ may be covered by finitely many closed balls of radius $r$. 
Furthermore, such a space is doubling if there exists $N\in\mathbb{N}$ independent of $x$ and $r$ such that $\overline{B}(x, 2r)$ can be covered by at most $N$ balls of radius $r$.
\end{dfn}

Letting $(X, \rho)$ be a metric space with the finite doubling property,  K\"aenm\"aki, Rajala and Suomala \cite{Krs} showed that there exists a nesting family of “cubes” which are similar as the dyadic cubes of Euclidean spaces.

\begin{thm}[\cite{Krs}]\label{thm:nest}
Let $(X,\rho)$ be a metric space with the finite doubling property and let $0<b<\frac{1}{3}$ be a constant. Then there exists a collection 
$\{Q_{n,i}\colon n\in\mathbb{Z}, i\in\mathbb{N}_n\subset \mathbb{N}\}$ of Borel sets that have the following properties:
\begin{enumerate}
	\item $X=\bigcup_{i\in\mathbb{N}_n}Q_{n,i}$ for every $n\in\mathbb{Z}$.
		\item $Q_{n,i}\cap Q_{m,j}=\varnothing $ or $Q_{n,i}\subset Q_{m,j}$, where $n,m\in \mathbb{Z}$, $n\ge m$, 
		$i\in \mathbb{N}_n $ and $j\in \mathbb{N}_m$.
		\item For every $n\in \mathbb{Z}$ and $i\in\mathbb{N}_n$, there exists a point $x_{n,i}\in X$ such that 
		\begin{equation}\label{cube}
		B(x_{n,i},c_1b^n)\subset Q_{n,i}\subset \overline{B}(x_{n,i},c_1'b^n),
\end{equation}
where $c_1=\frac{1}{2}-\frac{b}{1-b}$, $c'_1=\frac{1}{1-b}$.
\item There exists a point $x_0\in X$ so that for every $n\in\mathbb{Z}$, there is an index $i\in\mathbb{N}_n$ with 
$B(x_0,c_1b^n)\subset Q_{n,i}$.
\item $\{x_{n,i}\colon i\in \mathbb{N}_n\}\subset \{x_{n+1,i}\colon i\in \mathbb{N}_{n+1}\} $ for all $n\in\mathbb{Z}$.
	\end{enumerate}
\end{thm}

\begin{pro}\label{num}
Let $(X, \rho)$ be a metric space with an Ahlfors $s$-regular measure $\mu$.
\begin{enumerate}
	\item Suppose that $\mu$ is a probability measure, then for $n\in\mathbb{Z}$, we have $$C^{-1}{c'}_1^{-s}b^{-ns}\le\#\{Q_{n,i}\colon i\in\mathbb{N}_n\}\le Cc_1^{-s}b^{-ns}.$$
\item The metric space $(X,\rho)$ has the doubling property. 
\end{enumerate}
\end{pro}
\begin{proof}
(1) From the construction of $\{Q_{n,i} \colon n \in \mathbb{Z}, i \in \mathbb{N}_n\subset \mathbb{N}\}$, we see that for any
$n \in \mathbb{Z},~ Q_{n,i} \cap  Q_{n, j} = \varnothing$ for $i \ne j \in \mathbb{N}_n$ . Since $\mu$ is Ahlfors $s$-regular, combining (1), (3)
in Theorem \ref{thm:nest} and the equalities (\ref{ar}), we have
\begin{equation*}
\mu(X) =\sum_{i \in \mathbb{N}_n}\mu(Q_{n,i})\ge \sum_{i \in \mathbb{N}_n}\mu(B(x_{n,i},c_1b^n))\ge C^{-1}\#\mathbb{N}_n(c_1b^n)^s,
\end{equation*}
and 
\begin{equation*}
\mu(X) =\sum_{i \in \mathbb{N}_n}\mu(Q_{n,i})\le \sum_{i \in \mathbb{N}_n}\mu(\overline{B}(x_{n,i},c'_1b^n))\le C\#\mathbb{N}_n(c'_1b^n)^s,
\end{equation*}
which implies  that 
\[C^{-1}{c'}_1^{-s}b^{-ns}\le\#\mathbb{N}_n\le Cc_1^{-s}b^{-ns}.\]

(2) Given $x\in X$ and $r>0$, let $n_0=\min\{n\ge1\colon c'_1b^n< r\}$. Let $$I=\{i\in \mathbb{N}_{n_0}\colon Q_{n_0,i}\cap \overline{B}(x,2r)\ne\varnothing\}.$$ Then $\{\overline{B}(x_{n_0,i},r)\}_{i\in I}$ is a cover of  $\overline{B}(x,2r)$. Note that $\bigcup_{i\in I}Q_{n_0,i}\subset B(x,4r)$ and $r\le c'_1b^{n_0-1}$, then by (\ref{cube}), we have
\begin{equation*}
\begin{split}
 C4^sr^s&\ge\mu(B(x,4r))\ge\mu\Big(\bigcup_{i\in I}Q_{n_0,i}\Big)\\
 &=\sum_{i\in I}\mu(Q_{n_0,i})\ge \sum_{i\in I}\mu(B(x_{n_0,i},c_1b^{n_0}))\\
 &\ge C^{-1}c_1^sb^{n_0s}\#I\ge C^{-1}(bc_1/c'_1)^sr^s\#I,
 \end{split}
 \end{equation*}
which follows that $\#I\le C^2\Big(\frac{4c'_1}{bc_1}\Big)^s$.
\end{proof}

Therefore for a compact Ahlfors $s$-regular space $(X,\mathscr{B},\mu,\rho)$, there exists  the family $\{Q_{n,i} \colon n \ge0, i \in \mathbb{N}_n\}$ satisfying the properties in Theorem \ref{thm:nest}, which are called ``generalized dyadic cubes".  For convenience, we write $\mathcal{Q}_0=\{X\}$ and $\mathcal{Q}_n=\{Q_{n,i}\colon i\in \mathbb{N}_n\}$ for $n\ge1$, and $\Q=\bigcup_{n\ge0}\Q_n$.

\subsection{Large intersection properties in metric spaces}

\begin{dfn}\label{ms}
Let $(X, \rho)$ be a metric space. A Borel measure $\mu$ on $X$ is said to be doubling if it is finite and positive in every ball
and there exists a constant $1\le c_2<\infty$ such that for all $x \in X$ and $r > 0$,
\[0<\mu\Big(\overline{B}(x, 2r)\Big) \le c_2 \mu\Big(\overline{B}(x, r)\Big)<\infty. \]

\end{dfn}
Let $(X,\rho)$ be a metric space endowed with doubling measure $\mu$. Denote $\tau=\dim_{\rm H}X$, and $\tau$ is finite, see  \cite[Section 13]{fr20}. In \cite{ns}, Negreira and Sequeira gave the definition of the classes of $G_{\delta}$ sets with large intersection property. Recall that a $G_{\delta}$ set is a countable intersection of open sets.
\begin{dfn}\label{lip}
 Let $0<t\le\tau$. Given $F\subset X$, define the net content
\[\mathscr{M}_{\infty}^t(F)=\inf\Bigl\{\sum_{i\ge1}\mu( Q_i)^{t/\tau}\colon F\subset \bigcup_{i\ge1} Q_i~{\rm where ~}Q_i\in \Q\Bigr\}.\]
 We denote $\G^t(X)$ the class of all $G_{\delta}$ sets $F \subset X$ such that for any $t'<t$, 
\[\mathscr{M}_{\infty}^{t'}(F\cap Q)=\mathscr{M}_{\infty}^{t'}(Q)\]
holds for all $Q\in\Q$. 
\end{dfn}

\begin{rem}
(1) Negreira and Sequeira \cite{ns} showed that the net contents given by  different dyadic decomposition are equivalent, and pointed out that the net content is defined by $\mu(Q)^{t/\tau}$ instead of $(\diam Q)^{t}$, since the function $Q\mapsto(\diam Q)^t$ is not sub-additive if $t>1$ which is important in the study of $\mathscr{M}_{\infty}^t$ and large intersection property.

(2) Notice that the class $\mathcal{G}^t(X)$ given in Definition \ref{lip} depends on $\mu$. Since Ahlfors $s$-regular measures on $X$ are equivalent, therefore the classes $\mathcal{G}^t(X)$ defined by different Ahlfors $s$-regular measures are same. In this paper we only consider such measures, hence we use the notation $\mathcal{G}^t(X)$ instead of $\mathcal{G}_{\mu}^t(X)$.
\end{rem}

There are several equivalent definitions of $\mathcal{G}^t(X)$, see \cite{ns}. In this paper, we will use the following equivalent definition.
\begin{thm}[\cite{ns}]\label{ed}
Let $(X, \rho)$ be a metric space enowed with doubling measure $\mu$, and  $\tau = \dim_{\rm H} X$. Take
$F\subset X$ a $G_{\delta}$ subset and $0 < t \le \tau$. Then the following statements are equivalent:
\begin{enumerate}
	\item For all generalized dyadic cubes $Q\in\Q$, we have
\[\mathscr{M}_{\infty}^{t'}(F\cap Q)=\mathscr{M}_{\infty}^{t'}(Q)\quad(\forall t'<t).\] 
    \item There exists a constant $0 < c \le1$ such that for all generalized dyadic cubes $Q\in\Q$, we have
\be\label{ed}
\mathscr{M}_{\infty}^{t'}(F\cap Q)\ge c\mathscr{M}_{\infty}^{t'}(Q) \quad(\forall t'<t).
\ee
\end{enumerate}
\end{thm}

\begin{rem}
Note that, given $0 < t\le \tau $, $\mathscr{M}_{\infty}^t(Q)= \mu(Q)^{t/\tau}$ holds for all dyadic cubes $Q\in \Q$. Then (\ref{ed}) can be rewritten as 
\be\label{ned}
\mathscr{M}_{\infty}^{t'}(F\cap Q)\ge c\mu(Q)^{t'/\tau} \quad(\forall t'<t).
\ee
\end{rem}
The following are some properties of the class $\G^{t}(X)$.

\begin{pro}[\cite{ns}]\label{pro}
Let $(X, \rho)$ be a metric space endowed with doubling measure $\mu$, and $\tau = \dim_{\rm H} X$ and take $0<t\le \tau$.
\begin{enumerate}
\item If $0 < t_1\le t$, then $\G^{t}(X)\subset \G^{t_1}(X)$.
\item If $F \subset E \subset X$, where $E$, $F$ are $G_{\delta}$ sets, and $F \in \G^t(X)$, then $E\in \G^t(X)$.
\item If $(X, \rho)$ is complete, then $\G^t(X)$ is closed under countable intersections.
\item When $\mu$ is Ahlfors $\tau$-regular, then $F \in \G^{t}(X)$ implies that $\dim_{\rm H} F \ge t$.
\end{enumerate}
\end{pro}

\begin{rem}
Let $(X, \rho)$ be a metric space endowed with doubling measure $\mu$, and $\tau = \dim_{\rm H} X$. Take $0<t\le \tau$, then from the definition of $\G^{t}(X)$, we have
\[ \G^{t}(X)=\bigcap_{t'<t}\G^{t'}(X).\]
\end{rem}

\section{Proofs of Theorem \ref{main} and Corollary \ref{mainn}}
In 2019, Persson \cite{per19} used potentials and energies of measures to prove that some sets have large intersection property, which is  a new and useful method in the study of large intersection property. Later, Persson \cite{per21} applied this idea to show that the  limsup sets, generated by  sequences of open sets, belong to $\mathcal{G}^{\gamma}(\mathbb{T}^d)$ for some $\gamma>0$.
In this paper, we borrowed some ideas  from \cite{per21} to prove Theorem \ref{main}. Before that, we give some related lemmas.

 In this section, we always assume that $(X,\mathscr{B},\mu,\rho)$ is a probability space where the metric space $(X,\rho)$ is compact and $\mu$ is  an Ahlfors $s$-regular Borel measure  ($0<s<\infty$). Let $\Q=\bigcup_{k\ge0}\Q_k$ be the collection of generalized dyadic cubes in $(X, \rho)$ given in Section 2.1.
\begin{lem}\label{leme}
For $0\le t<s$ and $U\in \mathscr{B}$ with $\diam U>0$, we have
\begin{equation}\label{es}
(\diam U)^{-t}\mu(U)^2\le I_t(\mu, U)\le C_1(\diam U)^{s-t}\mu(U),
\end{equation}
where $C_1>0$ is an absolute constant.
\end{lem}

\begin{proof}
Write $l=\diam U$, $l_j=2^{-j}l$. For $t<s$ and $y\in U$, we have
\begin{equation*}
\begin{split}
 \phi_t(\mu,U,y)&=\int_{U} \rho(x,y)^{-t}d\mu(x) \le\int_{\overline{B}(y,l)} \rho(x,y)^{-t}d\mu(x) = \sum_{j\ge0}\int_{\overline{B}(y,l_j)\setminus \overline{B}(y,l_{j+1})}\rho(x,y)^{-t}d\mu(x)\\
&\le  \sum_{j\ge0}l_{j+1}^{-t}\mu\Big(\overline{B}(y, l_j)\setminus \overline{B}(y, l_{j+1})\Big)= \sum_{j\ge0}l_{j+1}^{-t}\Big(\mu(\overline{B}(y, l_j))-\mu( \overline{B}(y, l_{j+1}))\Big)\\
&\le \sum_{j\ge0}l_{j+1}^{-t}\Big(Cl_j^s-C^{-1}l_{j+1}^s\Big)= l^{s-t}\sum_{j\ge0}2^{(1+j)(t-s)} (C2^s-C^{-1})\\
&= C_1(\diam U)^{s-t},
\end{split}
\end{equation*}
where $C_1>0$ is a constant. Then $I_t(\mu, U)=\int_U  \phi_t(\mu,U,y)d\mu(y)\le C_1(\diam U)^{s-t} \mu(U)$. 
 Also
\[I_t(\mu, U)=\int_U\int_U \rho(x , y)^{-t}d\mu(x)d\mu(y)\ge \int_U\int_U (\diam U)^{-t}d\mu(x)d\mu(y)=(\diam U)^{-t}\mu(U)^2.\]

\end{proof}

\begin{rem} For $0\le t<s$, since $(X, \rho)$ is compact, we have $I_t(\mu)\le C_1\max\{1,(\diam X)^s\}$ $<\infty$.
\end{rem}

 \begin{lem}\label{lemiii}
 Let $\{\mu_n\}_{n\ge1}$ be a sequence of Borel measures on $X$ which are absolutely continuous with respect to $\mu$. Suppose there exists a constant $M>0$ such that for any  ball $B\subset X$,
\begin{equation}\label{lem1}
  M^{-1}\le \liminf_{n\to\infty}\frac{\mu_n(B)}{\mu(B)}\le \limsup_{n\to\infty}\frac{\mu_n(B)}{\mu(B)}\le M,
  \end{equation}
and
$\sup_{n\ge1}I_{\gamma}(\mu_n)<M$ for some $\gamma>0$. Then for $t<\gamma$, there exists a constant $M_1>0$ such that
\[M_1^{-1}\le\liminf_{n\to\infty}\frac{I_t(\mu_n, Q)}{I_t(\mu,Q)} \le \limsup_{n\to\infty}\frac{I_t(\mu_n, Q)}{I_t(\mu,Q)}\le M_1\]
holds for all $Q\in\mathcal{Q}$.
 \end{lem}

\begin{rem}
	In Lemma \ref{lemiii}, the condition $\sup_{n\ge1}I_{\gamma}(\mu_n)<M$ can be weaken to $\sup_{n\ge N_0}I_{\gamma}(\mu_n)$ $<M$ for some $N_0\in\mathbb{N}$.
\end{rem}

\begin{proof}
 Fix any $Q\in\mathcal{Q}$. For any $\alpha>0$, given $y\in Q$,  let $Q_{\alpha}(y)=\{x\in Q\colon \rho(x,y)<{\alpha}^{-1/\gamma}\}$. Then 
    \begin{equation}\label{qm}
\phi_t(\mu_n, Q,y)=\int_{Q_{\alpha}(y)}\rho(x,y)^{-t} d\mu_n(x)+\int_{Q\setminus Q_{\alpha}(y)}\rho(x,y)^{-t} d\mu_n(x).
  \end{equation}
Since $I_{\gamma}(\mu_n)<M$, we have
  $$M>\iint \rho(x , y)^{-\gamma} d\mu_n(x)d\mu_n(y) \ge \int_Q\Big(\int_{Q_{\alpha}(y)} \rho(x , y)^{-\gamma} d\mu_n(x)\Big)d\mu_n(y)> \alpha\int_Q\mu_n(Q_{\alpha}(y))d\mu_n(y),$$
  and it follows that
     \begin{equation}\label{qm1}
    \begin{split}
&\int_Qd\mu_n(y)\int_{Q_{\alpha}(y)}\rho(x,y)^{-t} d\mu_n(x)=\int_Qd\mu_n(y)\int_{Q_{\alpha}(y)}d\mu_n(x) \int_0^{\rho(x,y)^{-t}}du   \\
    &=\int_Qd\mu_n(y)\int_{Q_{\alpha}(y)} d\mu_n(x) \int_0^{\alpha^{t/\gamma}}du +\int_Qd\mu_n(y)\int_{Q_{\alpha}(y)}d\mu_n(x)\int_{\alpha^{t/\gamma}}^{\rho(x,y)^{-t}}du\\
    &=\alpha^{t/\gamma}\int_Q\mu_n(Q_{\alpha}(y))d\mu_n(y)+\int_{{\alpha}^{t/\gamma}}^{\infty}du\int_Q d\mu_n(y)\int_{Q_{u^{\gamma/t}}(y)}d\mu_n(x) \\
&\le M{\alpha}^{t/\gamma-1} +M\int_{{\alpha}^{t/\gamma}}^{\infty}u^{-\gamma/t}du\\
&= M{\alpha}^{t/\gamma-1}  +\frac{Mt}{(\gamma-t)}{\alpha}^{t/\gamma-1}=\frac{M\gamma}{(\gamma-t)}{\alpha}^{t/\gamma-1},
    \end{split}
  \end{equation} 
where we apply Fubini's theorem to derive the third equality. For $\epsilon>0$, for $\alpha$ large enough, $\int_Qd\mu_n(y)\int_{Q_{\alpha}(y)}\rho(x,y)^{-t} d\mu_n(x) <\epsilon$.
Observe that
	    \[\int_Q d\mu_n(y) \int_{Q\setminus Q_{\alpha}(y)}\rho(x,y)^{-t}d\mu_n(x)\le \int_Q\int_Q\min\{\rho(x,y)^{-t},\alpha^{t/\gamma}\}d\mu_n(x)d\mu_n(y).\]
     Denote $f(x,y)=\min\{\rho(x,y)^{-t},\alpha^{t/\gamma}\}$. Note that $f(x,y)$ is continuous on $Q\times Q$. Let $\{Q_{k,i}\in\mathcal{Q}_k: 1\le i\le N(k)\}$ be a partition of $Q$ by Borel sets satisfying $\diam Q_{k,i}\to 0$ as $k\to\infty$, and take $(x_i,y_j)\in Q_{k,i}\times Q_{k,j}$, then for all $(x,y)\in Q\times Q$ we have
     \[\sum_{i,j}^{N(k)}f(x_i,y_j)\chi_{\{Q_{k,i}\times Q_{k,j}\}}(x,y)\to f(x,y),\ as\ k\to\infty.\]
        Since $f(x,y)\le \alpha^{t/\gamma}$, using Dominated Convergence Theorem, we obtain
        \begin{equation*}
    	\begin{split}
    		 \int_Q\int_Qf(x,y)d\mu_n(x)d\mu_n(y)&=\lim_{k\to\infty} \int_Q\int_Q\sum_{i,j}^{N(k)}f(x_i,y_j)\chi_{\{Q_{k,i}\times Q_{k,j}\}}(x,y)d\mu_n(x)d\mu_n(y)\\
    		 &=\lim_{k\to\infty}\sum_{i,j}^{N(k)}f(x_i,y_j)\mu_n(Q_{k,i})\mu_n(Q_{k,j})\\
    		 &\le\lim_{k\to\infty}\sum_{i,j}^{N(k)}f(x_i,y_j)\mu_n(\overline{B}_{k,i})\mu_n(\overline{B}_{k,j}),
        \end{split}
        \end{equation*}
       where $\overline{B}_{k,i}=\overline{B}(x_{k,i},c_1'b^k)$. Recall that $B(x_{k,i},c_1b^k)\subset Q_{k,i}\subset\overline{B}_{k,i}$, there exists a constant $C_2>0$ such that $\mu(\overline{B}_{k,i})\le C_2\mu(Q_{k,i})$. We derive from (\ref{lem1}) that
       \[\limsup_{n\to\infty}\frac{\mu_n(\overline{B})}{\mu(\overline{B})}\le C^2M.\]
       Then for any $\theta>0$, for $n$ large enough, we have
       \begin{equation*}
     	\begin{split}
    		\lim_{k\to\infty}\sum_{i,j}^{N(k)}f(x_i,y_j)\mu_n(\overline{B}_{k,i})\mu_n(\overline{B}_{k,j}) &\le(C^2M+\theta)^2\lim_{k\to\infty}\sum_{i,j}^{N(k)}f(x_i,y_j)\mu(\overline{B}_{k,i})\mu(\overline{B}_{k,j})\\
    		 &\le(C^2M+\theta)^2C_2^2\lim_{k\to\infty}\sum_{i,j}^{N(k)}f(x_i,y_j)\mu(Q_{k,i})\mu(Q_{k,j})\\
    		 &=(C^2M+\theta)^2C_2^2\int_Q\int_Qf(x,y)d\mu(x)d\mu(y)\\
    		 &\le(C^2M+\theta)^2C_2^2\int_Q\int_Q\rho(x,y)^{-t}d\mu(x)d\mu(y).
    	\end{split}
       \end{equation*}
       Since $\epsilon>0$ is arbitrary, we prove that
       \[\limsup\limits_{n\to\infty}\frac{I_t(\mu_n, Q)}{I_t(\mu,Q)}\le C^4M^2C_2^2.\]  

 Notice that $I_t(\mu_n,Q)\ge (\diam Q)^{-t}\mu_n(Q)^2$ and $\mu_n(Q)\ge C^{-1}2^{-s}(c_1/c_1')^s(\diam Q)^s$, then combining these with Lemma \ref{leme},
  there exists an absolute constant $C_3>0$ such that 
$$\frac{I_t(\mu_n, Q)}{I_t(\mu, Q)}\ge C_3.$$
Let $M_1=\max\{C_3^{-1},C^4M^2C_2^2\}$, then we proved this lemma.
  \end{proof}

\begin{lem}\label{lemi}
Under the same conditions as Lemma \ref{lemiii},  let $\{A_n\}_{n\ge1}$ be a sequence of open sets in $X$ satisfying $\mu_n(X\setminus A_n) = 0$, $n\ge1$. Then  $\limsup\limits_{n\to\infty} A_n \in \mathcal{G}^{\gamma}(X)$.
 \end{lem}

 \begin{proof}
 Let $Q\in\mathcal{Q}$ and $t<\gamma$. 
 By (\ref{lem1}), $\mu_n(Q)>0$ for $n$ large enough. Define 
  \[\nu_n(F)=\int_{F\cap Q}\phi_t^{-1}(\mu_n,Q,y)d\mu_n(y)\]
for any $F\in\mathscr{B}$. Since $\mu_n(X\setminus A_n)=0$, $\nu_n(Q)=\nu_n(Q\cap A_n)$. Given $\epsilon>0$, combining Jensen's inequality \cite{jens} and Lemma \ref{lemiii}, for $n$ large enough,  we have
\begin{equation}\label{lower}
\begin{split}
\nu_n(F)&=\mu_n(F\cap Q)\int_{F\cap Q}\Big( \int_Q \rho(x,y)^{-t} d\mu_n(x)\Big)^{-1}\frac{1}{\mu_n(F\cap Q)}d\mu_n(y)\\
&\ge \mu_n(F\cap Q)\Big(\int_{F\cap Q} \Big(\int_Q \rho(x,y)^{-t} d\mu_n(x)\Big) \frac{1}{\mu_n(F\cap Q)} d\mu_n(y)\Big)^{-1}\\
&\ge (\mu_n(F\cap Q))^2\Big(\int_{Q} \int_Q \rho(x,y)^{-t}d\mu_n(y) d\mu_n(x)\Big)^{-1}\\
& >(\mu_n(F\cap Q))^2(M_1+\epsilon)^{-1}\Big(I_t(\mu,Q)\Big)^{-1}.
\end{split}
\end{equation}
Hence  $\nu_n$ is a nonzero measure and absolutely continuous with respect to $\mu_n$.  By Lemma \ref{leme} and (\ref{lower}), for $n$ large enough, we obtain
\be\label{nu1}
\nu_n(Q)\ge (M_1+\epsilon)^{-1}\frac{(\mu_n(Q))^2}{I_t(\mu,Q)}
\ge C_1^{-1}(M_1+\epsilon)^{-1}\frac{(\mu_n(Q))^2}{(\diam Q)^{s-t}\mu(Q)}= c_3\mu(Q)^{t/s},
\ee
where $c_3>0$ independ of $Q$ and $n$. By Jensen's inequality,  
\[\phi_t^{-1}(\mu_n,Q,y) =\frac{1}{\mu_n(Q)}\Big(\int_Q \frac{\rho(x,y)^{-t}}{\mu_n(Q)}d\mu_n(x)\Big)^{-1}\le \frac{1}{(\mu_n(Q))^2}\int_Q \rho(x,y)^{t}d\mu_n(x) 
.\]

Then for any $J\in\mathcal{Q}$ and $J\subset Q$, 
we get that
\be\label{nu2}
\begin{split}
\nu_n(J)&=\int_{J} \phi_t^{-1}(\mu_n,Q,y) d\mu_n(y)\le\int_{J} \phi_t^{-1}(\mu_n,J,y) d\mu_n(y)\\
&\le \frac{1}{(\mu_n(J))^2}\int_{J} \int_J \rho(x,y)^{t}d\mu_n(x) d\mu_n(y)\\
&\le (\diam J)^t\le c_4\mu(J)^{t/s}
\end{split}
\ee
for $n$ large enough and $c_4>0$ is an absolute constant.

  Let $\{Q_k\}_{k\ge1}$ be a cover of $Q\cap A_n$, where $Q_k\in\mathcal{Q}$. Without loss of generality,  we assume $Q_k\cap Q_j=\varnothing$ for $k\ne j$ and $Q_k\subset Q$. For $n$ large enough, we derive from (\ref{nu1}) and (\ref{nu2}) that
  \[ \sum_{k}\mu( Q_k)^{t/s}\ge c_4^{-1}\sum_k\nu_n(Q_k)\ge c_4^{-1}\nu_n(Q\cap A_n)=c_4^{-1}\nu_n(Q)\ge c_3c_4^{-1}\mu(Q)^{t/s}.\]
 It follows that 
 \[ \liminf_{n\to\infty}\mathscr{M}_{\infty}^t(Q\cap A_n)\ge c_3c_4^{-1}\mu(Q)^{t/s}\] 
 for $n$ large enough. Therefore, for any $m\ge1$ and generalized dyadic cube $Q$,
\[\mathscr{M}_{\infty}^t\Big(\bigcup_{n\ge m}A_n\cap Q\Big)\ge \sup_{n\ge m}\mathscr{M}_{\infty}^t(A_n\cap Q)\ge c_3c_4^{-1} \mu(Q)^{t/s} .\quad(\forall t<\gamma)\]
So $\bigcup_{n\ge m}A_n\in\mathcal{G}^{\gamma}(X)$ by (\ref{ned}). Thus from (3) in Proposition \ref{pro}, we prove that
$$\limsup_{n\to\infty}A_n=\bigcap_{m\ge 1}\bigcup_{n\ge m}A_n\in\mathcal{G}^{\gamma}(X).$$
 
  \end{proof}

Recall the 5r-covering theorem in metric space.
\begin{lem}[\cite{ma}]\label{vitali}
Let $(X, \rho)$ be a separable metric space and $\mathcal{A}$ be a family either of closed balls or open balls such that
\[\sup\{\diam B\colon B\in\mathcal{A}\}<\infty.\]
Then there is a finite or countable sequence $\{B_i\}_{i\in I}$ of pairwise disjoint balls such that
\[\bigcup_{B\in\mathcal{A}}B\subset \bigcup_{i\in I}5B_i.\]
\end{lem}

Now we prove Theorem \ref{main}.

\noindent\textbf{Proof of Theorem \ref{main}}: 
Since $\mu\Big(\limsup\limits_{n\to\infty}B_n\Big)=1$, for $n\ge1$, there is some $N_n$ such that 
$$\mu\Big(\bigcup_{m=n}^{N_n}B_m\Big)>1-\frac{1}{2n}.$$
Denote $A_n=\bigcup_{m=n}^{N_n}B_m$. By Lemma \ref{vitali}, there is $\mathcal{I}_n\subset \{n, n+1,\dots N_n\}$ such that 
  \[B_i\cap B_j=\varnothing, {\rm ~for~}i\ne j\in\mathcal{I}_n\quad{\rm and}\quad A_n\subset\bigcup_{m\in\mathcal{I}_n}5B_m.\]
(i) Construct measures
 
 Let $A'_n=\bigcup_{m\in\mathcal{I}_n}B_m$. Since 
 \[\mu(A'_n)=\sum_{m\in\mathcal{I}_n}\mu(B_m)\ge C^{-2}5^{-s}\sum_{m\in\mathcal{I}_n}\mu(5B_m)>C^{-2}5^{-s}(1-\frac{1}{2n})>0,\]
and $\mu(A'_n)<\infty$, we define $$\eta_n(F)=\frac{1}{\mu(A'_n)}\mu(F\cap A'_n)$$ for any  $F\in\mathscr{B}$. Then $\eta_n$ is a nonzero probability measure supported on $A'_n$ satisfying $\eta_n(A'_n)=\eta_n(X)=1$.
  
 Let $F_n=\bigcup_{m\in\mathcal{I}_n}E_m$. Define
 \[\mu_n(F)=\sum_{m\in\mathcal{I}_n}\eta_n(B_m)\frac{\mu(E_m\cap F)}{\mu(E_m)}\]
 for $F\in\mathscr{B}$. Then $\mu_n$ is a nonzero probability measure supported on $F_n$ satisfying $\mu_n(F_n)=\mu_n(X)=1$.
 
Let $B\subset X$ be a ball. Firstly we will show the following inequalities

\be\label{nu}
C^{-4}5^{-s}\le \liminf_{n\to\infty}\frac{\eta_n(B)}{\mu(B)}\le \limsup_{n\to\infty}\frac{\eta_n(B)}{\mu(B)}\le C^45^s.
\ee
Denote $B=B(x_B,r_B)$. Since $r_n$ decreases to 0 as $n\to\infty$, for any $\theta>0$, there is some $N>0$ such that for $n\ge N$, $m\in  \mathcal{I}_n$, we have $r_m<r_B\theta/2$. Note that 
$$\bigcup_{m\in \mathcal{I}_n\atop B\cap B_m\ne\varnothing}B_m\subset (1+\theta)B,$$
hence 
\[
\eta_n(B)=\frac{1}{\mu(A'_n)}\mu(B\cap A'_n)\le\frac{1}{\mu(A'_n)}\mu\Big( \bigcup_{m\in \mathcal{I}_n\atop B\cap B_m\ne\varnothing}B_m\Big)\le 
\frac{\mu((1+\theta)B)}{\mu(A'_n)}<\frac{C^2(1+\theta)^s\mu(B)}{C^{-2}5^{-s}(1-1/2n)},
\]
 which implies 
$$\limsup_{n\to\infty}\frac{\eta_n(B)}{\mu(B)}\le C^45^s.$$
 Let $\mathcal{I}_n(B)=\{m\in \mathcal{I}_n\colon B_m\subset B\}$, then we have
\begin{equation*}
\frac{\eta_n(B)}{\mu(B)}\ge  \frac{1}{\mu(A'_n)\mu(B)}\mu\Big( \bigcup_{m\in \mathcal{I}_n(B)}B_m\cap A'_n\Big)\ge 
\frac{1}{\mu(B)}\sum_{m\in \mathcal{I}_n(B)}\mu( B_m).
\end{equation*}
We claim that 
 \be\label{cla}
 \liminf\limits_{n\to\infty}\sum_{m\in\mathcal{I}_n(B)}\mu(B_m)\ge C^{-4}5^{-s}\mu(B).
 \ee
  Then
\begin{equation*}
\liminf_{n\to\infty}\frac{\eta_n(B)}{\mu(B)}\ge\liminf_{n\to\infty}\frac{1}{\mu(B)}\sum_{m\in \mathcal{I}_n(B)}\mu( B_m)\ge 
C^{-4}5^{-s}.
\end{equation*}
Now we prove the claim. For any $0<\epsilon<1/2$, take $n$ large enough such that $\mu(B)>C^22^{s-1}/n$, and 
 $r_m<\epsilon r_B/10,~m\in \mathcal{I}_n$. Then if $m\in \mathcal{I}_n\setminus \mathcal{I}_n(B)$, that is, $B^c\cap B_m\ne\varnothing$, we have $5B_m\cap (1-\epsilon)B=\varnothing$. Therefore
\begin{equation*}
\begin{split}
 \mu\Big((1-\epsilon)B\cap A_n\Big)&\le \mu\Big(\bigcup_{m\in \mathcal{I}_n(B)}5B_m\Big)+\mu\Big(\bigcup_{m\in \mathcal{I}_n\setminus \mathcal{I}_n(B)}5B_m\cap (1-\epsilon)B\Big)\\
 &=\mu\Big(\bigcup_{m\in \mathcal{I}_n(B)}5B_m\Big)\le \sum_{m\in \mathcal{I}_n(B)}\mu(5B_m)\le C^25^s\mu\Big(\bigcup_{m\in \mathcal{I}_n(B)}B_m\Big).
 \end{split}
 \end{equation*}
We also note that
\begin{equation*}
 \mu\Big((1-\epsilon)B\cap A_n\Big)\ge  \mu\Big((1-\epsilon)B\Big)-\frac{1}{2n}\ge C^{-2}(1-\epsilon)^s\mu(B)-\frac{1}{2n}.
 \end{equation*}
Combining the inequalities above, we prove the claim.

Secondly we will prove that
\begin{equation*}\label{mu}
 C^{-8}5^{-2s}\le \liminf_{n\to\infty}\frac{\mu_n(B)}{\mu(B)}\le \limsup_{n\to\infty}\frac{\mu_n(B)}{\mu(B)}\le C^65^s
 \end{equation*}
hold for any ball $B$. Given any $\theta>0$, Since  $\bigcup_{m\in\mathcal{I}_n\atop B_m\cap B\ne\varnothing}B_m\subset (1+\theta)B$ holds for $n$ large enough, we obtain
 \begin{equation*}
\begin{split}
\mu_n(B)&=\sum_{m\in\mathcal{I}_n}\eta_n(B_m)\frac{\mu(E_m\cap B)}{\mu(E_m)}= \sum_{m\in\mathcal{I}_n\atop B_m\cap B\ne\varnothing}\eta_n(B_m)\frac{\mu(E_m\cap B)}{\mu(E_m)}\\
&\le\sum_{m\in\mathcal{I}_n\atop B_m\cap B\ne\varnothing}\eta_n(B_m)\le \eta_n((1+\theta)B)\le (C^45^s+\theta)(1+\theta)^{s}C^2\mu(B).
\end{split}
\end{equation*}
The last inequality follows from (\ref{nu}). So $\limsup\limits_{n\to\infty}\frac{\mu_n(B)}{\mu(B)}\le C^65^s$.

We derive from (\ref{nu}) and (\ref{cla}) that 
\begin{equation*}
\begin{split}
\liminf_{n\to\infty}\frac{\mu_n(B)}{\mu(B)}&=\liminf_{n\to\infty}\frac{1}{\mu(B)}\sum_{m\in\mathcal{I}_n}\eta_n(B_m)\frac{\mu(E_m\cap B)}{\mu(E_m)}\\
&\ge \liminf_{n\to\infty}\frac{1}{\mu(B)}\sum_{m\in\mathcal{I}_n(B)}\eta_n(B_m)\\
&\ge \liminf_{n\to\infty}\frac{C^{-4}5^{-s}}{\mu(B)}\sum_{m\in\mathcal{I}_n(B)}\mu(B_m)\\
&\ge C^{-8}5^{-2s}.
\end{split}
\end{equation*}
(ii) Show that $\sup_{n\ge N}I_t(\mu_n)<\infty$ for some $N\ge 1$
 
Note that 
 \begin{equation*}
\begin{split}
I_t(\mu_n)&=\int_X\int_X \rho(x , y)^{-t}d\mu_n(x)d\mu_n(y)=\int_{F_n}\int_{F_n} \rho(x , y)^{-t}d\mu_n(x)d\mu_n(y)\\
&=\sum_{m\in\mathcal{I}_n}\sum_{i\in\mathcal{I}_n}\int_{E_m}\int_{E_i} \rho(x , y)^{-t}d\mu_n(x)d\mu_n(y).
\end{split}
\end{equation*}
When $i=m$, since $\mu_n|_{E_m}=\frac{\eta_n(B_m)}{\mu(E_m)}\mu|_{E_m}$, we have
 \begin{equation*}
\begin{split}
& \int_{E_m}\int_{E_m} \rho(x , y)^{-t}d\mu_n(x)d\mu_n(y)=\Big(\frac{\eta_n(B_m)}{\mu(E_m)}\Big)^2\int_{E_m}\int_{E_m} \rho(x , y)^{-t}d\mu(x)d\mu(y)\\
&=\Big(\frac{\eta_n(B_m)}{\mu(E_m)}\Big)^2 I_t(\mu,E_m)\le (C^45^s+\theta)^2\Big(\frac{\mu(B_m)}{\mu(E_m)}\Big)^2 I_t(\mu,E_m)
\end{split}
\end{equation*} 
 for $n$ large enough. Due to $t<\lambda$, there is some absolute constant $M_3>0$ such that 
\[\sup_{n\ge1}\frac{I_t(\mu, E_n)\mu(B_n)}{\mu(E_n)^2}<M_3. \]
Hence 
 \begin{equation}\label{same}
\int_{E_m}\int_{E_m} \rho(x , y)^{-t}d\mu_n(x)d\mu_n(y)\le (C^45^s+\theta)^2\Big(\frac{\mu(B_m)}{\mu(E_m)}\Big)^2 I_t(\mu,E_m)
\le (C^45^s+\theta)^2M_3\mu(B_m).
\end{equation}

For $i\ne m$,
we suppose that there is a ball $cB_m ~(0<c<1)$ such that $E_m\subset cB_m$ for any $m\ge1$. Otherwise, we consider $c^{-1}B_m$ instead of $B_m$ in the previous proof. For $x\in E_m,~y\in E_i$, we get
$$\rho(x,y)\ge \rho(x_i,x_m)-c(r_m+r_i)\ge (1-c)\rho(x_i,x_m).$$ 
Then
\begin{equation*}
\begin{split}
& \int_{E_m}\int_{E_i} \rho(x , y)^{-t}d\mu_n(x)d\mu_n(y)\le \int_{E_m}\int_{E_i} ((1-c)\rho(x_i,x_m))^{-t}d\mu_n(x)d\mu_n(y)\\
&=((1-c)\rho(x_i,x_m))^{-t}\mu_n(E_m)\mu_n(E_i)=((1-c)\rho(x_i,x_m))^{-t}\eta_n(B_m)\eta_n(B_i)\\
&\le (1-c)^{-t}(C^45^s+\theta)^2\rho(x_i,x_m)^{-t} \mu(B_m)\mu(B_i)\\
&\le (1-c)^{-t}(C^45^s+\theta)^22^t\int_{B_m}\int_{B_i} \rho(x,y)^{-t}d\mu(x)d\mu(y).
\end{split}
\end{equation*}  
The last inequality follows from $\rho(x,y)\le \rho(x_i,x_m)+r_i+r_m\le 2\rho(x_i,x_m)$ for $x\in B_m,~y\in B_i$.

Therefore by Lemma \ref{leme} and the fact that $\{B_m\}_{m\in\mathcal{I}_n}$ are disjoint balls, for $n$ large enough, we obtain
\begin{equation*}
\begin{split}
I_t(\mu_n)& \le (C^45^s+\theta)^2\sum_{m\in\mathcal{I}_n}\Big(\frac{\mu(B_m)}{\mu(E_m)}\Big)^2I_t(\mu,E_m)\\
&+\frac{(C^45^s+\theta)^22^t}{(1-c)^{t}} \sum_{m\in\mathcal{I}_n}\sum_{i\in\mathcal{I}_n\atop i\ne m}\int_{B_m}\int_{B_i}\rho(x,y)^{-t}d\mu(x)d\mu(y)\\
&\le (2C^45^s)^2M_3\sum_{m\in\mathcal{I}_n}\mu(B_m)+(1-c)^{-t}(2C^45^s)^22^tI_t(\mu)\\
&\le (2C^45^s)^2M_3+(1-c)^{-t}(2C^45^s)^22^tC_1\max\{1,(\diam X)^{s}\}<\infty.
\end{split}
\end{equation*}  
By Lemma \ref{lemi}, $\limsup\limits_{n\to\infty} F_n\in \mathcal{G}^t(X)$ for all $t<\lambda$, hence $\limsup\limits_{n\to\infty} F_n\in \mathcal{G}^{\lambda}(X)$ which implies that $\limsup\limits_{n\to\infty} E_n\in \mathcal{G}^{\lambda}(X)$.  $\hfill\square$

\noindent\textbf{Proof of Corollary \ref{mainn}}: 
Without loss of generality, we assume that $\diam B_n<1$. Given $t\in(0,s]$, write $F_n=B_n^t$, then $\mu\Big(\limsup\limits_{n\to\infty}F_n\Big)=1$.  For $\beta\ge0$, by Lemma \ref{leme},  we have
$$C^{-1}2^{-\beta}r_n^{t-\beta}\le\frac{I_{\beta}(\mu,B_n)\mu(F_n)}{\mu(B_n)^2}\le C^2C_12^{s-\beta}r_n^{t-\beta}
.$$
Notice that if $\beta\le t$, we have $\sup_{n\ge1} \frac{I_{\beta}(\mu,B_n)\mu(F_n)}{\mu(B_n)^2}<\infty$, and if $\beta>t$, then $\limsup\limits_{n\to\infty} \frac{I_{\beta}(\mu,B_n)\mu(F_n)}{\mu(B_n)^2}=\infty$, since $r_n$ decreases to 0 as $n\to\infty$. Hence 
\[t=\sup\Big\{\beta\ge0\colon \sup_{n\ge1} \frac{I_{\beta}(\mu,B_n)\mu(F_n)}{\mu(B_n)^2}<\infty\Big\},\]
and we applying Theorem \ref{main} to have $\limsup\limits_{n\to\infty}B_n\in\mathcal{G}^t(X)$.      $\hfill\square$

\section{Applications}
Let $(X,\rho)$ be a compact metric space endowed with an Ahlfors $s$-regular probability measure $\mu$. In this section we investigate the large intersection property of  limsup random fractals, random covering sets, and limsup sets generated by rectangles.
\subsection{Limsup random fractals}
\ 
\newline
\\
\indent
Given $0<b<1/3$, there is a family of generalized dyadic cubes of $X$ defined in Subsection 2.1. In this subsection, we consider a model of limsup random fractals which is constructed by using generalized dyadic cubes $\Q=\bigcup_{n\ge0}\Q_n$.

 For $n\ge1$, let $\Big\{Z_n(Q),Q\in \mathcal{Q}_n\Big\}$ be a sequence of random variables, each taking values in $\{0,1\}$. 
Let
\[A(n)=\bigcup_{Q\in \mathcal{Q}_n,\atop Z_n(Q)=1}{Q^o},
\]
where $Q^o$ is the interior of $Q$. The random set 
\[A=\limsup_{n\to\infty}A(n)\]
is called a limsup random fractal associated to $\big\{Z_n(Q),Q\in \mathcal{Q}_n,n\ge 1\big\}$. 

For $n\ge1$, and $Q\in \mathcal{Q}_n$, denote the probability $P_n(Q):=\mathbb{P}(Z_n(Q)=1)$, and 
\begin{align}
 \gamma_1\colon&=-\limsup_{n\to\infty}\frac{\max_{Q\in\mathcal{Q}_n}\log_{b^{-1}}P_n(Q) }{n}, \label{22}\\
 \gamma_2\colon&=-\limsup_{n\to\infty}\frac{\min_{Q\in\mathcal{Q}_n}\log_{b^{-1}}P_n(Q) }{n}. \label{23}
\end{align}
We refer to $\gamma_1$ and $\gamma_2$ as the indices of the limsup random fratcal $A$.

We assume the following condition, which is similar as the condition (H2) in \cite{HLX}.

${\bf Correlation~ Condition}$:\, Suppose that there is a constant $\delta\ge0$ such that for all $\epsilon>0$, 
\begin{equation}\label{ast}
	\limsup_{n\to\infty}\frac{1}{n}\log_{b^{-1}}f(n,\epsilon)\le\delta,
\end{equation}
where 
\[f(n,\epsilon)=\max_{Q\in \mathcal{Q}_n}\#\{Q'\in \mathcal{Q}_n\colon{\rm Cov}(Z_n(Q),Z_n(Q'))\ge\epsilon P_n(Q)P_n(Q')\}.\]

\begin{thm}\label{limsup}
Let $(X,\mathscr{B},\mu,\rho)$ be a compact Ahlfors $s$-regular space $(0<s<\infty)$.   Let $A=\limsup\limits_{n\to\infty}A(n)$ be a limsup random set with indices $\gamma_1$, $\gamma_2$ and satisfy the Correlation Condition. If $\gamma_2+\delta<s$, then $A\in \G^{s-\gamma_2-\delta}(X)$ a.s.
\end{thm}

\begin{cor}\label{eslim}
Under the setting in Theorem \ref{limsup}, we have
   \[\max\{0,s-\gamma_2-\delta\}\le \dim_{\rm H}A\le \max\{0,s-\gamma_1\}\quad a.s.\]
\end{cor}
\begin{rem}
We can also use the hitting probability of limsup random fractals to get the estimation on Hausdorff dimension in Corollary \ref{eslim}.  For
instance, let $A$ be a limsup random fratcal in $\mathbb{T}$ with $\gamma_1=\gamma_2<1$, $\delta=0$, and if $P_n(Q)$ is independent of $Q$, Khoshnevisan, Peres and Xiao \cite{kpx} showed that $\dim_{\rm H}A=1-\gamma_1$ a.s.  For more related research, see \cite{hcl, HLX}.
\end{rem}
Before proving  Theorem \ref{limsup}, we give some lemmas. For $n\ge1$ and $Q\in\Q_n$, recalling  (\ref{cube}), there is some $x_Q\in Q$ such that 
\[B(x_Q,c_1b^n)\subset Q\subset \overline{B}(x_Q,c'_1b^n),\]
and here $c_1=\frac{1}{2}-\frac{b}{1-b}$ and $c'_1=\frac{1}{1-b}$.
\begin{lem}\label{larger}
	Let $C$ and $c_1'$ be the constants in (\ref{eqah}) and (\ref{cube}) respectively. Let $B=B(x_B,r)$ be a ball with $r>0$ and
	\[\Q_n(B)=\{Q\in\Q_n\colon Q\subset B\}.\] 
	Then there exists some $N\ge1$ such that  for all $n\ge N$,  we have
	 $$\#\Q_n(B)\ge C^{-2}(r{c'}_1^{-1}b^{-n}-2)^s.$$
\end{lem}

\begin{proof}
	For $Q\in\Q_n(B)$, there exists $x_Q\in Q$ such that
	\[Q\subset \overline{B}(x_Q,c_1'b^n).\]
 For $r>0$, there exists some $N\ge1$ such that  for all $n\ge N$,  we have $r>2c_1'b^n$. Then for $Q\in\Q_n$, observe that if $Q\notin\Q_n(B)$, then $Q\cap B(x_B,r-2c_1'b^n)=\varnothing$, and we deduce that
	\[B(x_B,r-2c_1'b^n)\subset\bigcup_{Q\in\Q_n(B)}Q\cup \bigcup_{Q\notin\Q_n(B)\atop Q\cap B\ne\varnothing}Q=\bigcup_{Q\in\Q_n(B)}Q.\]
	Hence
	\begin{equation*}
	\begin{split}
	C^{-1}(r-2c_1'b^n)^s&\le \mu\big(B(x_B,r-2c_1'b^n)\big)\le\mu\Big(\bigcup_{Q\in\Q_n(B)}Q\Big)\le\mu\Big(\bigcup_{Q\in\Q_n(B)}\overline{B}(x_Q,c_1'b^n)\Big)\\
&\le\#\Q_n(B)C(c_1'b^n)^s,
\end{split}
\end{equation*}
which derives that 
    \[\#\Q_n(B)\ge\frac{(r-2c_1'b^n)^s}{C^2(c_1'b^n)^s}\quad(\forall n\ge N).\]
\end{proof}

Recall the notation $B^t(x,r)=B(x,r^{t/s})$.
 \begin{lem}\label{lemt}
	Suppose that $\gamma_2+\delta<s$. For $t>0$ with $t<s-\gamma_2-\delta$, let 
	$$B^{t}(n)=\bigcup_{Q\in\mathcal{Q}_n\atop Z_n(Q)=1}B_Q^{t},$$
	where $B_Q=B(x_Q,c_1b^n)\subset Q$. Given $x\in X$, let $J_{n,x}$ be the unique dyadic cube in ${\Q_n}$ containing $x$.
	Then for $k\ge1$, we have
	\[\P\Big(\bigcup_{n\ge k}\{J_{n,x}\subset B^{t}(n)\}\Big)=1.\]
\end{lem}

\begin{proof}
	Given $x\in X$, there is a $x'\in J_{n,x}$ such that
	\[B(x',c_1b^n)\subset J_{n,x}.\]
    Let
	\[\mathcal{B}_n(x)=\big\{J\in\mathcal{Q}_n\colon J\subset B(x',(c_1b^n)^{t/s})\big\}.\]
	By Lemma \ref{larger}, for $n$ large enough, we have 
	\be\label{bn}
	\#\mathcal{B}_n(x)\ge \frac{1}{C^2}\Big(\frac{c_1^{t/s}}{c_1'}b^{n(t/s-1)}-2\Big)^s.
	\ee
	Define $M_n(x)=\sum_{J\in\B_n(x)}Z_n(J)$, and note that
	\[\big\{ J_{n,x}\subset B^{t}(n) \big\}\supset\big\{Z_n(J)=1{\rm ~for ~some~}J\in \B_n(x)\big\} =\big\{M_n(x)>0\big\}.\]
	 From Chebyshev's inequality, we obtain
	\begin{equation*}
		\P(M_n(x)=0)\le \frac{{\rm  Var}(M_n(x))}{(\mathbb{E}(M_n(x)))^2}.
	\end{equation*}
	Since ${\rm Cov}(Z_n(J),Z_n(J'))\le\mathbb{E}(Z_n(J)Z_n(J'))\le P_n(J)$, we have
	\begin{equation}
		\begin{split}
			{\rm Var}(M_n(x))&=\sum_{J\in\B_n(x)}\sum_{J'\in\B_n(x)}{\rm Cov}(Z_n(J),Z_n(J'))\\
			&=\sum_{J\in\B_n(x)}\Big(\sum_{J'\in\G_n(J)} {\rm Cov}(Z_n(J),Z_n(J'))+\sum_{J'\in\mathcal{T}_n(J)}{\rm Cov}(Z_n(J),Z_n(J'))\Big)\\
			&\le \sum_{J\in\B_n(x)}\Big(\sum_{J'\in\G_n(J)} \epsilon P_n(J)P_n(J')+\sum_{J'\in\mathcal{T}_n(J)}P_n(J)\Big)\\
			&\le \epsilon\Big(\sum_{J\in\B_n(x)}P_n(J)\Big)\Big(\sum_{J'\in\G_n(J)}P_n(J')\Big)+\max_{J\in\B_n(x)}\#\mathcal{T}_n(J)\Big(\sum_{J\in\B_n(x)}P_n(J)\Big),
		\end{split}
	\end{equation}
	where
    \begin{equation*}
    \begin{split}
     \G_n(J)&=\{J'\in\B_n(x)\colon {\rm Cov}(Z_n(J),Z_n(J'))<\epsilon P_n(J)P_n(J')\},\\
    \mathcal{T}_n(J)&=\B_n(x)\setminus\G_n(J). 
    \end{split}
    \end{equation*}
     Recalling the notation of the Correlation Condition, we obtain
    \[{\rm Var}(M_n(x))\le \epsilon\Big(\sum_{J\in\B_n(x)}P_n(J)\Big)^2+f(n,\epsilon)\Big(\sum_{J\in\B_n(x)}P_n(J)\Big).\]
    Note that $\mathbb{E}(M_n(x))=\sum_{J\in\B_n(x)}P_n(J)\ge \#\B_n(x)\big(\min_{Q\in\mathcal{Q}_n}P_n(Q)\big)$, we derive that
\begin{equation}
   \begin{split}
	\P(M_n(x)=0)&\le \frac{\epsilon(\mathbb{E}(M_n(x)))^2+f(n,\epsilon)(\mathbb{E}(M_n(x)))}{(\mathbb{E}(M_n(x)))^2}\\
	&=\epsilon+\frac{f(n,\epsilon)}{\mathbb{E}(M_n(x))}\\
	&\leq \epsilon+\frac{f(n,\epsilon)}{\#\B_n(x)\big(\min_{Q\in\mathcal{Q}_n}P_n(Q)\big)}.
   \end{split}
\end{equation}
    From (\ref{ast}) and (\ref{bn}), for $\theta>0$ with $2\theta<s-\delta-\gamma_2-t$, we get
    \[f(n,\epsilon)\le b^{-n(\delta+\theta)}\quad \text{and}\quad \#\mathcal{B}_n(x)\ge c_5b^{n(t-s)}\]
    for $n$ large enough, where $c_5>0$ is an absolute constant, and from (\ref{23}), we have
    \[\min_{Q\in\mathcal{Q}_n}P_n(Q)\ge b^{n(\gamma_2+\theta)}\]
    for infinitely many $n$, denoted by $\mathcal{N}$.
    Then
    \[\limsup_{ n\in \mathcal{N}\atop n\to\infty}\frac{f(n,\epsilon)}{\#\B_{n}(x)\min_{Q\in\mathcal{Q}_{n}}P_{n}(Q)}\le \limsup_{ n\in \mathcal{N}\atop n\to\infty}c_5^{-1}b^{-n(2\theta+\delta+\gamma_2-s+t)}=0,\]
    which implies $\limsup\limits_{ n\in \mathcal{N}\atop n\to\infty}\P(M_n(x)=0)=0.$
    Therefore
    \[\limsup_{ n\in \mathcal{N}\atop n\to\infty}\P(M_n(x)>0)=1.\]
    So we show that
    \[\limsup_{n\to\infty}\P\big( J_{n,x}\subset B^{t}(n) \big)=1.\]
    Hence for $k\ge 1$,
    \[\mathbb{P}\Big(\bigcup_{n= k}^{\infty}\{J_{n,x}\subset B^{t}(n)\}\Big)\ge \limsup_{n\to\infty}\P\big(J_{n,x}\subset B^{t}(n)\big)= 1.\]
\end{proof}

\noindent\textbf{Proof of Theorem \ref{limsup}}: 
Let $x\in X$ and $t>0$ with $t<s-\gamma_2-\delta$.
By Lemma \ref{lemt}, we have $\P\Big(\bigcup_{n= k}^{\infty}\{J_{n,x}\subset B^{t}(n)\}\Big)=1$ for any $k\ge1$. Note that
\[\Big\{x\in\bigcup_{n=k}^{\infty}B^{t}(n)\Big\}\supset \big\{\exists~ n\ge k {\rm~such ~that~} J_{n,x}\subset B^{t}(n)\big\}.\]
This yields that $\P\big\{x\in\bigcup_{n=k}^{\infty}B^{t}(n)\big\}=1$. Write the set $F_{k}=\bigcup_{n=k}^{\infty}B^{t}(n)$, then 
\[\int d\mu(x)\int\chi_{\{x\in F_k\}}d\P(\omega) =1.   \]
Using Fubini's Theorem, we obtain
\[\int d\P(\omega)\int\chi_{F_k}(x)d\mu(x) =1,   \]
implying that $\mu(F_k)=1$ a.s. for any $k\ge1$. Hence $\mu\Big(\limsup\limits_{n\to\infty}B^{t}(n)\Big)=1$ a.s.

By Corollary \ref{mainn}, for $t>0$ with $t<s-\gamma_2-\delta$, we have $\limsup\limits_{n\to\infty}B(n)\in \mathcal{G}^{t}(X)$ a.s., which implies $\limsup\limits_{n\to\infty}B(n)\in \mathcal{G}^{s-\gamma_2-\delta}(X)$ a.s. Hence by Proposition \ref{pro}, $\limsup\limits_{n\to\infty}A(n)\in \mathcal{G}^{s-\gamma_2-\delta}(X)$ a.s.    $\hfill\square$

\noindent\textbf{Proof of Corollary \ref{eslim}}:
If $s\ge\gamma_1$, let $t>s-\gamma_1$. For $\delta>0$, there exists $k_0\ge1$ such that $2c_1'b^{k_0}\le \delta<2c_1'b^{k_0-1}$. Since
 $$A\subset \bigcup_{n= k}^{\infty}\bigcup_{Q\in\Q_n\atop Z_n(Q)=1}Q^o$$
  for all $k\ge1$, we have 
\[\mathcal{H}^t_{\delta}(A)\le \sum_{n\ge k_0}\sum_{Q\in\Q_n}(\diam Q)^tZ_n(Q).\]
Then
\[ \mathbb{E}(\mathcal{H}^t_{\delta}(A))\le \sum_{n\ge k_0}\sum_{Q\in\Q_n}(\diam Q)^tP_n(Q).  \]
Considering $\theta>0$ with $\gamma_1>\theta>s-t$, by (\ref{22}), for $n$ large enough, 
\[\max_{Q\in\Q_n}P_n(Q)\le b^{n\theta},\]
which implies  that 
\be\label{hau}
\begin{split}
 \mathbb{E}(\mathcal{H}^t_{\delta}(A))&\le \sum_{n\ge k_0}\#\Q_n(2c'_1b^n)^tb^{n\theta}\\
 &\le Cc_1^{-s}(2c'_1)^t\sum_{n\ge k_0}b^{n(t-s+\theta)}<\infty, 
 \end{split}
 \ee
giving $\dim_{\rm H} A \le  s-\gamma_1$ a.s.  If $\gamma_1>s$, taking $t>0$ and $\theta=s$ in (\ref{hau}), we have $ \mathbb{E}(\mathcal{H}^t_{\delta}(A))\le Cc_1^{-s}(2c'_1)^t\sum_{n\ge k_0}b^{nt}<\infty$ a.s.,  which implies $\dim_{\rm H} A \le 0$ a.s.

When $s>\gamma_2+\delta$, from Theorem \ref{limsup}, we have $A\in\mathcal{G}^{s-\gamma_2-\delta}(X)$ a.s. Then we obtain from Proposition \ref{pro} (4) that $\dim_{\rm H} A \ge  s-\gamma_2-\delta$ a.s.   $\hfill\square$

\subsection{Random covering sets}
\ 
\newline
\\
\indent
Durand \cite{du10} showed that random covering sets in $\mathbb{T}$, defined as limsup sets of a sequence of balls whose centers are independent and uniformly distributed  random variables, have large intersection property.  In this subsection, we consider random covering sets in $(X,\mathscr{B},\mu,\rho)$ with weak dependence condition.
 Let $\{\xi_n\}_{n\ge1}$ be a stationary process on a probability space $(\Omega,\mathcal{B},\mathbb{P})$ having $\mu$ as probability law.
\begin{dfn}\label{emp}
We say that $\{\xi_n\}_{n\ge1}$ is exponentially mixing if for any $n\ge1$, there exist two constants $c>0$ and $0<\gamma<1$ such that  
\[\big|\mathbb{P}(\xi_1\in A|D)-\mathbb{P}(\xi_1\in A)\big|\le c\gamma^{n}\]
holds for any ball $A\subset X$ and $D\in\mathcal{B}_{n+1}$, where $\mathcal{B}_{n+1}$ is the sub-$\sigma$-field generated by $\{\xi_{n+i}\}_{i\ge1}$.
\end{dfn}

Let $\{r_n\}_{n\ge1}$ be a sequence of positive real numbers decreasing to zero. 
The $random$ $covering$ $set$ is defined as
\[E\colon=\limsup_{n\to\infty}B(\xi_n,r_n)=\{y\in X\colon y\in B(\xi_n,r_n) ~\text{for~ infinitely ~many}~n\ge1\}.\]
The set $E$ consists of the points which are covered by $\{B(\xi_n,r_n)\}_{n\ge1}$ infinitely often.

\begin{thm}[\cite{HL}]\label{hl}
	Let $\{\xi_n\}_{n\ge1}$ be exponentially mixing and the probability law $\mu$ be Ahlfors $s$-regular. Then we have
	\[\mu\bigl(E\bigr)=
	\begin{cases}
	0&  \text{if $\sum_{n=1}^{\infty}r_n^s<\infty$}\\
	1&  \text{if $\sum_{n=1}^{\infty}r_n^s=\infty$}
	\end{cases} \quad a.s.\]
\end{thm}

\begin{thm}\label{random}
Let $(X,\mathscr{B},\mu,\rho)$ be a compact Ahlfors $s$-regular space $(0<s<\infty)$. Let $\{\xi_n\}_{n\ge1}$ be exponentially mixing and the probability law $\mu$ be Ahlfors $s$-regular, then the random covering set $E\in \mathcal{G}^{s_0}(X)$ a.s., where $s_0=\inf\{t\ge0\colon \sum_{n=1}^{\infty}r_n^t<\infty\}$.
\end{thm}

\begin{proof}
 For $t<s_0$, by Theorem \ref{hl}, we have 
\[\mu(\limsup_{n\to\infty} B(\xi_n,r_n^{t/s}))=1\quad a.s.,\]
 since $\sum_{n=1}^{\infty}(r_n^{t/s})^s=\infty$. Applying Corollary \ref{mainn}, we finish the proof.
\end{proof}

\subsection{ Limsup sets generated by rectangles}
\ 
\newline
\\
\indent
The mass transference principle from balls to limsup sets generated by rectangles was
established by \cite{wwx} in 2015.  In 2021, Ding \cite{ding} extended the result to product compact metric spaces, and proved that such sets have large intersection property. In this subsection, we use a different method to show the large intersection property of limsup sets generated by rectangles. 

For $1\le i\le d$, let $(X_i, \rho_i)$ be a compact metric spaces equipped with a Borel probability measure $\mu_i$ which is Ahlfors $s_i$-regular, that is, there exists a constant $\delta_i>1$ such that $\delta_i^{-1}r^{s_i}\le\mu_i(B_i(x,r))\le \delta_ir^{s_i}$, here $B_i(x, r)$ is the open ball in $X_i$. We consider the metric space $(\prod_{i=1}^{d}X_i,\rho)$ with the measure $\mu$, where
\[\rho(x,y)= \max_{1\le i\le d}\{\rho_i(x_i, y_i)\}\]
for $x=(x_1,\dots, x_d),~ y=(y_1,\dots, y_d)\in \prod_{i=1}^{d}X_i$ and $\mu$ satisfies that $\mu(\prod_{i=1}^{d}A_i)=\prod_{i=1}^{d}\mu_i(A_i)$ for $A_i$ is a Borel set in $X_i$.
The metric space $(\prod_{i=1}^{d}X_i,\rho)$ is compact.

\begin{thm} 
For $1\le i\le d$, let $\{B_{i}(x_{n,i},r_n)\}_{n\ge1}$ be a sequence of balls in $X_i$ with $r_n\to0$ as $n\to\infty$. Let $\{a_i\}_{1\le i\le d}$ be a sequence of positive numbers with $1 \le a_1 \le\dots\le a_d$. Assume that $\mu(\limsup\limits_{n\to\infty}\prod_{i=1}^{d}B_{i}(x_{n,i},r_n))=1$. 
Then $\limsup\limits_{n\to\infty} \prod_{i=1}^{d}B_{i}(x_{n,i},r_n^{a_i}) \in \mathcal{G}^s(\prod_{i=1}^{d}X_i)$,
where 
\[s = \min_{1\le i\le d}\Big\{ \frac{\sum_{j=1}^d s_j+a_i\sum_{j=1}^is_j-\sum_{j=1}^ia_js_j}{a_i}\Big\}.\]
\end{thm}
\begin{proof}
We assume that $r_n<1$, for $n\ge1$. Denote $B_n=\prod_{i=1}^{d}B_{i}(x_{n,i},r_n)$ and $R_n=\prod_{i=1}^dB_{i}(x_{n,i},r_n^{a_i})$. 
 Let $x=(x_1,\dots,x_d)$ and $y=(y_1,\dots,y_d)$. For $t<s$ and $y\in R_n$, we have
\begin{equation*}
\phi_t(\mu,R_n,y)=\int_{R_n}\rho(x,y)^{-t}d\mu(x)\le \int_{\prod_{i=1}^dB_{i}(y_i, 2r_n^{a_i})}\rho(x,y)^{-t}d\mu(x).
\end{equation*}
Let $j$ be the integer with $\sum_{i=1}^{j-1}s_i< t\le \sum_{i=1}^j s_i$. Denote
\begin{align*}
	& F_{1,j}=\Big\{x\in \prod_{i=1}^dB_{i}(y_i, 2r_n^{a_i})\colon \max_{1\le i\le j}\rho_i(x_i,y_i)<4r_n^{a_j}\Big\},\\
	& F_{2,j}=\Big\{x\in \prod_{i=1}^dB_{i}(y_i, 2r_n^{a_i})\colon \max_{1\le i\le j-1}\rho_i(x_i,y_i)\ge r_n^{a_j}\Big\}.
\end{align*}
Note that $R_n\subset  F_{1,j}\cup F_{2,j}$. Using Fubini's theorem, we derive that
\begin{equation}\label{dif}
\begin{split}
&\phi_t(\mu,R_n,y)\le \int_{F_{1,j}}\rho(x,y)^{-t}d\mu(x)+\int_{F_{2,j}}\rho(x,y)^{-t}d\mu(x)\\
&\le \int_{F_{1,j}}(\max_{1\le i\le j}\rho_i(x_i,y_i))^{-t}d\mu_1(x_1)\dots d\mu_d(x_d)+\int_{F_{2,j}}(\max_{1\le i\le j-1}\rho_i(x_i,y_i))^{-t}d\mu_1(x_1)\dots d\mu_d(x_d)\\
&\le \beta_1\prod_{i=j+1}^dr_n^{a_is_i}\int_{\prod_{i=1}^jB_{i}(y_i, 4r_n^{a_j})}(\max_{1\le i\le j}\rho_i(x_i,y_i))^{-t}d\mu_1(x_1)\dots d\mu_j(x_j)\\
&+\beta_1\prod_{i=j}^dr_n^{a_is_i}\int_{\prod_{i=1}^{j-1}B_{i}(y_i, 2r_n^{a_i})\setminus \prod_{i=1}^{j-1}B_{i}(y_i, r_n^{a_j})}  (\max_{1\le i\le j-1}\rho_i(x_i,y_i))^{-t}  d\mu_1(x_1)\dots d\mu_{j-1}(x_{j-1})\\
&\le \beta_1'\Big(\prod_{i=j+1}^dr_n^{a_is_i}r_n^{-a_jt}\prod_{i=1}^jr_n^{a_js_i}+\prod_{i=j}^dr_n^{a_is_i}I_2\Big),
\end{split}
\end{equation}
where $\beta_1,~\beta_1'$ are absolute constants. Writting $f=(\max\limits_{1\le i\le j-1}\rho_i(x_i,y_i))^{-t} $, we have
\begin{equation}\label{diff}
\begin{split}
I_2&=\int_{\prod_{i=1}^{j-1}B_{i}(y_i, 2r_n^{a_i})\setminus \prod_{i=1}^{j-1}B_{i}(y_i, r_n^{a_j})} \Big(\int_{0}^fdu\Big)d\mu_1(x_1)\dots d\mu_{j-1}(x_{j-1})\\
&=\int_{\prod_{i=1}^{j-1}B_{i}(y_i, 2r_n^{a_i})\setminus \prod_{i=1}^{j-1}B_{i}(y_i, r_n^{a_j})} r_n^{-a_jt}d\mu_1(x_1)\dots d\mu_{j-1}(x_{j-1})\\
& -\int_{\prod_{i=1}^{j-1}B_{i}(y_i, 2r_n^{a_i})\setminus \prod_{i=1}^{j-1}B_{i}(y_i, r_n^{a_j})} \Big(\int_f^{r_n^{-a_jt}}du\Big)d\mu_1(x_1)\dots d\mu_{j-1}(x_{j-1})\\
&\le r_n^{-a_jt}\Big(\prod_{i=1}^{j-1}\mu_i(B_{i}(y_i, 2r_n^{a_i}))-\prod_{i=1}^{j-1}\mu_i(B_{i}(y_i, r_n^{a_j}))\Big)\\
&- \int_0^{r_n^{-a_jt}}du\int_{\prod_{i=1}^{j-1}B_{i}(y_i, 2r_n^{a_i})\setminus \prod_{i=1}^{j-1}B_{i}(y_i, u^{-1/t})} d\mu_1(x_1)\dots d\mu_{j-1}(x_{j-1})\\
&=-r_n^{-a_jt}\prod_{i=1}^{j-1}\mu_i(B_{i}(y_i, r_n^{a_j}))+\int_0^{r_n^{-a_jt}} \prod_{i=1}^{j-1}\mu_i(B_{i}(y_i, u^{-1/t}))du\\
&\le-r_n^{-a_jt}\prod_{i=1}^{j-1}\delta_i^{-1}r_n^{a_js_i}+\prod_{i=1}^{j-1}\delta_i\int_0^{r_n^{-a_jt}} u^{-\sum_{i=1}^{j-1}s_i/t}du\le \alpha r_n^{-a_jt}\prod_{i=1}^{j-1}r_n^{a_js_i},
\end{split}
\end{equation}
where $\alpha>0$ is an absolute constant. Therefore by (\ref{dif}) and (\ref{diff}), 
 \begin{equation*}
\phi_t(\mu,R_n,y)\le \beta_1'(1+\alpha)\prod_{i=j+1}^dr_n^{a_is_i}r_n^{-a_jt}\prod_{i=1}^jr_n^{a_js_i}.
\end{equation*}
 
Hence there exists an absolute constant $\beta_2>0$ such that
\begin{equation}\label{it}
\begin{split}
\frac{I_t(\mu,R_n)\mu(B_n)}{\mu(R_n)^2}\le \beta_2r_n^{-a_jt+a_j\sum_{i=1}^js_i+\sum_{i=1}^ds_i-\sum_{i=1}^ja_is_i}\le\beta_2.
\end{split}
\end{equation}
From the proof of Theorem \ref{main}, for any $t$ such that $\sup_{n\ge1}\frac{I_t(\mu,R_n)\mu(B_n)}{\mu(R_n)^2}<\infty$, we have $\limsup\limits_{n\to\infty}R_n\in\mathcal{G}^t(\prod_{i=1}^{d}X_i)$. Note that (\ref{it}) holds for any $t<s$, then $\limsup\limits_{n\to\infty}R_n \in \mathcal{G}^s(\prod_{i=1}^{d}X_i)$.

\end{proof}

\subsection*{Acknowledgements}
The work was supported by  NSFC 11671151 and 11871208.

 { }

\end{document}